\documentclass{article}

\usepackage{amsmath}
\usepackage{amsthm}
\usepackage{amsfonts}
\usepackage{amsbsy}
\usepackage{amssymb}
\newtheorem{theorem}{Theorem}
\newtheorem{proposition}{Proposition}
\newtheorem{corollary}{Corollary}

\newcommand{\R}{{\mathbb R}}

\newcommand{\set}[2]{ \left\{ #1 \ \left| \ #2 \right. \right\} }

\title{On multilinear determinant functionals}
\author{Philip T. Gressman\footnote{Partially supported by NSF grant DMS-0850791.}}
\begin{document}
\maketitle

\begin{abstract}
This paper considers the problem of $L^p$-estimates for a certain multilinear functional involving integration against a kernel with the structure of a determinant.  Examples of such objects are ubiquitous in the study of Fourier restriction and geometric averaging operators.  It is shown that, under very general circumstances, the boundedness of such functionals is equivalent to a geometric inequality for measures which has recently appeared in work by D. Oberlin \cite{oberlin2000II} and Bak, Oberlin, and Seeger \cite{bos2008}.
\end{abstract}

\section{Introduction}
Let $\mu$ be a $\sigma$-finite, nonnegative Borel measure on some real, separable Hilbert space ${\cal H}$ (in practice, ${\cal H} = \R^d$ with the usual inner product).  For any positive integer $k$, let $\det(y_1,\ldots,y_{k+1})$ equal $k!$ times the Euclidean $k$-dimensional volume of the simplex with vertices $y_1,\ldots,y_{k+1}$ (this volume is simply the square-root of the determinant of the $k \times k$ matrix of Gram type whose $(i,j)$-entry is the inner product $\left< y_i-y_{k+1},y_j-y_{k+1} \right>_{\cal H}$).
The purpose of this paper is to characterize boundedness of the multilinear functional
\begin{equation}
T^{-\gamma}_{\mu,k}(f_1,\ldots,f_{k+1}) := \int \cdots \int \prod_{j=1}^{k+1} f_j(y_j) \frac{d \mu(y_{1}) \cdots d \mu(y_{k+1})}{(\det(y_1,\ldots,y_{k+1}))^{\gamma}} \label{mainobject}
\end{equation}
on $(k+1)$-fold products of Lebesgue spaces $L^{p}(\mu)$. 
To avoid ambiguities in the definition of \eqref{mainobject}, it will be assumed that the $\mu$-measure of any affine $(k-1)$-dimensional subspace of ${\cal H}$ is zero; any nonnegative, $\sigma$-finite Borel measure satisfying this condition will be called $k$-admissible.  Under this assumption, the zero set $\set{(y_1,\ldots,y_{k+1}) \in {\cal H}^{k+1}}{ \det(y_1,\ldots,y_{k+1}) = 0}$ has measure zero with respect to the $(k+1)$-fold product measure $\mu^{k+1}$ and \eqref{mainobject} is well-defined.

Functionals of the form \eqref{mainobject} have been widely studied in harmonic analysis for many years.  Such objects originally arose in the context of the restriction phenomenon for the Fourier transform: the Hausdorff-Young inequality may be applied to an appropriate power of the extension operator at the cost of a singular change of variables which gives rise to \eqref{mainobject}.  Arguments of this type go back to Fefferman \cite{fefferman1970} and Zygmund \cite{zygmund1974}; Carleson and Sj\"{o}lin \cite{cs1972} were the first to explicitly study objects like \eqref{mainobject}, followed by Prestini \cite{prestini1979}, Christ \cite{christ1985}, and Drury and Marshall \cite{dm1985}, \cite{dm1987}, and many others.  In all these cases, the measure $\mu$ was the arclength or affine arclength measure of some finite-type curve in $\R^d$.  For many years, the study of \eqref{mainobject} beyond measures on curves was somewhat limited; to major examples of work falling in this category are Christ's \cite{christ1984} sharp results for \eqref{mainobject} when $\mu$ is the Lebesgue measure, and Drury's \cite{drury1988} variant of \eqref{mainobject} when $\mu$ is the surface measure on the unit sphere in $\R^d$.  Recent years have seen the reappearance of objects like \eqref{mainobject}, in large part due to the work of Christ \cite{christ1998} demonstrating a connection between these objects and restricted weak-type estimates for geometric averaging operators.  As in the case of the restriction problem, the work in this area is extensive; some representative examples include the work of D. Oberlin \cite{oberlin2000II}, Tao and Wright \cite{tw2003}, Erdo\v{g}an and R. Oberlin \cite{eo2008}, and the author \cite{Gressman2007}.

It has recently been observed by D. Oberlin that there is a natural geometric condition on measures which appears to play an important role in both the regularity properties of averages over hypersurfaces (Radon-like operators) and in the Fourier restriction problem.  For example, in \cite{oberlin2000II}, Oberlin proved that the operation of convolution with a surface measure (with smooth density and compactly supported) in $\R^d$ maps $L^{\frac{d+1}{d}}(\R^d)$ to $L^{d+1}(\R^{d})$ if and only if the measure of any ambient parallelepiped in $\R^d$ is bounded above by a constant times the parallelepiped's volume raised to the power $\frac{d-1}{d+1}$.  Likewise, Bak, Oberlin, and Seeger \cite{bos2008} prove a similar result in the context of Fourier restriction to certain classes of degenerate curves.  As both of these problems are approached using multilinear determinant-type functionals, it is reasonable to suspect that the geometric condition of Oberlin generalizes to this setting (for measures including the curve and hypersurface cases).  In this paper it will be demonstrated that this suspicion is correct and that the connection between \eqref{mainobject} and the geometric measure criterion of Oberlin is, in fact, very natural and general.

For simplicity, let a subset $B \subset {\cal H}$ be called an ellipsoid when it may be written as
\[ B := \set{ x \in {\cal H}}{\sum_i \frac{|\left< x-x_0, \omega_i\right>|^2}{\ell_i^2}  \leq 1 } \]
for some $x_0 \in {\cal H}$, some orthonormal basis $\{ \omega_i \}$ of $\cal H$, and lengths $\ell_i \in [0,\infty]$.  If $x_0 = 0$, the ellipsoid will be called centered.  Given an ellipsoid $B \subset {\cal H}$ and a positive integer $k$ with $k \leq \dim {\cal H}$, we define the $k$-content of $B$, denoted $|B|_k$, to be equal to the supremum of $\ell_{i_1} \cdots \ell_{i_k}$ as the indices $i_1 < i_2 < \cdots < i_k$ range over all possible values (with the convention that the product is zero if any of the individual lengths is zero, even if additional lengths are infinite).  
Now suppose $\mu$ is a nonnegative Borel measure on ${\cal H}$.  This measure will be called $k$-curved with exponent $\alpha > 0$ when there exists a finite constant $C_\alpha$ for which \begin{equation}
\mu(B) \leq C_\alpha |B|_k^\alpha \label{curved}
\end{equation} for all ellipsoids $B$.  Oberlin's parallelepiped condition, for example, is equivalent to the assertion that the measure $\mu$ on $\R^k$ is $k$-curved with exponent $\alpha$ (in particular, it suffices to test $\mu$ on ellipsoids rather than all parallelepipeds).  If the inequality \eqref{curved} holds for centered ellipsoids, then $\mu$ will be called $k$-curved at the origin.  Note that any $\mu$ which is $k$-curved for some exponent $\alpha$ is automatically $k$-admissible, implying that \eqref{mainobject} is well-defined.

The condition \eqref{curved} may be regarded as measuring the extent to which $\mu$ is supported on some $k$-dimensional subspace of ${\cal H}$.  For example, the Lebesgue measure on any affine $d$-dimensional subspace of ${\cal H}$ will be $k$-curved if and only if $k \leq d$, in which case the exponent will equal $\frac{d}{k}$.  More generally, if $\mu$ is any measure of smooth density on a compact piece of a $d$-dimensional submanifold of ${\cal H}$, the measure $\mu$ will be $k$-curved for some $k > d$ if and only if no small neighborhoods of the submanifold lie in an affine $(k-1)$-dimensional hyperplane modulo an infinite order perturbation.  There are a variety of similar geometric considerations which have appeared in the literature which also imply the inequality \eqref{curved}; some of these will be taken up in section \ref{conditionsec}.

The main theorem of this paper is that, for a certain range of exponents $\gamma$, the $k$-curvature inequality for $\mu$ is equivalent to boundedness of \eqref{mainobject} on appropriate products of $L^p(\mu)$-spaces.  In particular:
\begin{theorem}
\label{equivtheorem} For any $\alpha > 0$ and $k$-admissible $\mu$, the following are equivalent: 
\begin{enumerate}
\item There is a constant $C_\alpha < \infty$ such that \eqref{curved} holds
for all ellipsoids $B \subset {\cal H}$.
\item For any $\gamma \in (0,\alpha)$ and any exponents $p_i \in [1,\infty]$, $i=1,\ldots,k+1$, satisfying $1- \frac{1}{p_i} < \frac{\gamma}{k \alpha}$ for all $i$ and $\sum_{i=1}^{k+1} (1-\frac{1}{p_i}) = \frac{\gamma}{\alpha}$, the inequality 
\begin{equation}
 \left| T^{-\gamma}_{\mu,k} (f_1,\ldots,f_{k+1}) \right|  \leq C \prod_{i=1}^{k+1} ||f_i||_{L^{p_i}(\mu)} \label{multi}
\end{equation}
holds with a finite constant independent of the functions $f_i$.
\item For some $\gamma \in (0,\infty)$ and exponents $p_i \in [1,\infty]$ satisfying $\sum_{i=1}^{k+1} (1-\frac{1}{p_i}) = \frac{\gamma}{\alpha}$, the inequality \eqref{multi} holds when each $f_i$ is a characteristic function.
\item For any $\gamma \in (0,\infty)$, and any exponents $p_i \in [0,1]$ satisfying $\frac{1}{p_i} - 1 < \frac{\gamma}{k \alpha}$ and $\sum_{i=1}^{k+1} (\frac{1}{p_i} - 1 ) = \frac{\gamma}{\alpha}$, the inequality
\begin{equation}
T^{\gamma}_{\mu,k} (|f_1|,\ldots,|f_{k+1}|) \geq c \prod_{i=1}^{k+1} ||f_i||_{L^{p_i}(\mu)} \label{multi2}
\end{equation}
holds with a positive constant $c$ independent of the functions $f_i$.
\item For some $\gamma \in (0,\infty)$ and exponents $p_i \in [0,1]$ satisfying $\sum_{i=1}^{k+1} ( \frac{1}{p_i} - 1 ) = \frac{\gamma}{\alpha}$, the inequality \eqref{multi2} holds when each $f_i$ is a characteristic function.
\end{enumerate}
\end{theorem}

The proof of this theorem is contained in section \ref{proofsec}.  A key point to be established in the course of the proof is that the sublevel-type functional (where the exponent $\gamma$ is taken to be zero, but the integration in \eqref{mainobject} is restricted to the set where the determinant is below some threshold in magnitude) satisfies a certain restricted weak-type inequality.  This is accomplished by an inductive decomposition of the support of $\mu$ combined with a bootstrapping argument, similar to arguments of Bak, Oberlin, and Seeger \cite{bos2008} in the case of degenerate curves.  Section \ref{conditionsec} is devoted to some consequences of theorem \eqref{mainobject} which relate to various alternate geometric conditions appearing in the literature.  Before these topics are taken up, however, there are a few points regarding theorem \ref{equivtheorem} which are reasonable to address:
\begin{enumerate}
\item The range of $\gamma$ and exponents $p_i$ for which \eqref{multi} is shown to hold is the sharp range over which such inequalities can hold uniformly for all measures satisfying an Oberlin-type curvature condition like \eqref{curved}.  This fact is a consequence of Christ's sharp estimates for \eqref{mainobject} in the case where $\mu$ is the Lebesgue measure \cite{christ1984}.  On the other hand, it can and does happen that the sharp range for a single measure $\mu$ may be strictly greater than the range indicated by theorem \ref{equivtheorem}.
\item On the unit sphere $S^{n-1}$ with $k=n$, the inequality \eqref{rwt1} is the restricted weak-type version of the estimate of Drury \cite{drury1988}.  Passage to the full inequality can be accomplished by a Fubini argument coupled with the observation that the push-forward of the Lebesgue measure on $S^{n-1}$ to a hyperplane through the origin (via orthogonal projection) is itself $(n-1)$-curved with exponent $\alpha = 1$.
\item The passage from \eqref{mainobject} to condition \eqref{curved} afforded by theorem \ref{equivtheorem} is not only important because of the direct, geometric interpretation of the condition \eqref{curved}, but also because this condition may be formulated as a multilinear sublevel set problem on a product of one-dimensional spaces {\it a la} Phong, Stein, and Sturm \cite{pss2001}.  This technology is significantly more advanced than its higher dimensional analogues.   In particular, the success of the proof of theorem \ref{equivtheorem} relies on the ability to reduce the associated higher-dimensional sublevel set problem to a uniform family of one-dimensional sublevel set problems.
\end{enumerate}


\section{The proof of equivalence}
\label{proofsec}
The proof of theorem \ref{equivtheorem} will be accomplished by means of the multilinear trick of Christ \cite{christ1985}.  In particular, it will be shown that
\[ \left|T^{-\gamma}_{\mu,k}(\chi_{E_1},\ldots,\chi_{E_{k+1}}) \right| \leq C (|E_1| \cdots |E_{k+1}|)^{\frac{1}{p}} |E_{k+1}| \]
for an appropriate exponent $p$; the symmetry of the functional then allows for interpolation which will give \eqref{multi}.  This estimate will itself be established by means of a Fubini-type argument, in which $y_{k+1}$ is regarded as fixed and the resulting $k$-linear functional is considered.  To this end, let $\mu_1,\ldots,\mu_k$ be Borel probability measures on ${\cal H}$, and let $I_{k}^\delta(\mu_1,\ldots,\mu_k)$ be the $(\mu_1 \times \cdots \times \mu_k)$-measure of the set
\[ \set{(y_1,\ldots,y_k) \in {\cal H}^k}{0 < \det(0,y_1,\ldots,y_k) < \delta} \]
(note that by translation-invariance of \eqref{mainobject}, it suffices to fix $y_{k+1} = 0$; correspondingly it will suffice to discuss only centered ellipsoids).  The first and most important step towards the proof of theorem \ref{equivtheorem} is the following result, of certain interest in its own right, which relates the magnitude of $I^\delta_k(\mu,\ldots,\mu)$ to the measures of centered ellipsoids.  It holds even when $\mu$ is not assumed to be $k$-admissible (since $I_k^{\delta}$ is explicitly constructed to avoid the zero set of the $\det$ function):
\begin{theorem}
For any positive integer $k$, there are constants $c_k$ and $C_{k}$ such that for any $\epsilon > 0$ and any Borel probability measure $\mu$ on ${\cal H}$, 
\[ I^{c_k \delta}_{k}(\mu,\ldots,\mu) \leq C_{k} \epsilon \]
for any $\delta$ which is less than or equal to the infimum of the $k$-content of all centered ellipsoids $B$ satisfying $\mu(B) \geq \epsilon$. \label{mainst}
\end{theorem}
\begin{proof}
The proof is by induction on $k$ and bootstrapping.  To begin, let ${\cal C}_{k,\epsilon}$ be the supremum of $I_{k}^{c_k \delta}(\mu,\ldots,\mu)$ as $\mu$ ranges over all Borel probability measures on ${\cal H}$ and all $\delta$ satisfying $\delta \leq |B|_k$ for all $B$ with $\mu(B) \geq \epsilon$ (and $c_k := 2^{-(k-1)(k+2)/2}$).  Clearly ${\cal C}_{k,\epsilon} \leq 1$ for any values of $k,\epsilon$.  It will be shown by induction that ${\cal C}_{k,\epsilon} \leq C_k \epsilon$ as well.  Note that it suffices to assume $\epsilon \leq 1$.

Consider the case $k=1$: $\det(0,x) = ||x||$, so $I^\delta_{1}(\mu)$ equals the $\mu$-measure of the set of points whose distance to the origin is strictly between $0$ and $\delta$.  In particular, if $B$ is any metric ball centered at the origin which satisfies $\mu(B) \geq \epsilon$, it is trivially true that $I^{|B|_1}_1(\mu) \leq \epsilon$.  Thus, if $c_1 = 1$, it follows that ${\cal C}_{1,\epsilon} \leq \epsilon$.

For general $k$, fix $\epsilon > 0$ and let $r_0$ be the supremum of all $r$ which satisfy $\mu(\set{x \in {\cal H}}{||x|| \geq r}) \geq \epsilon$.  By dominated convergence, it must be the case that this supremum is attained, i.e., $\mu(\set{x \in {\cal H}}{|x| \geq r_0}) \geq \epsilon$ while $\mu(\set{x \in {\cal H}}{||x|| > r_0}) \leq \epsilon$.  Let $\mu_0$ be the (unique) measure of mass exactly $\epsilon$ which is obtained by taking a convex linear combination of the restrictions of $\mu$ to the sets $\set{x \in {\cal H}}{||x|| \geq r_0}$ and $\set{x \in {\cal H}}{||x|| > r_0}$, respectively.  For any nonzero $x$ in the support of $\mu_0$,
\begin{equation}
\det(0,x,y_1,\ldots,y_{k-1}) = ||x|| \det(0, P_{\hat x} y_1,\ldots, P_{\hat x} y_{k-1}), 
\label{volumefactor}
\end{equation}
where $P_{\hat x}$ is the orthogonal projection onto the subspace orthogonal to $x$.  Let $\mu_{\hat x}$ be the push-forward of $\mu$ via $P_{\hat x}$.  By the induction hypothesis, it follows that
\[ I_{k-1}^{c_{k-1}\delta'} (\mu_{\hat x}, \ldots, \mu_{\hat x}) \leq 2 C_{k-1} \epsilon \]
provided that $\delta'$ is smaller than the $(k-1)$-content of any centered ellipsoid $B_{\hat x}$ in the Hilbert space ${\cal H}_{\hat x}$ (which is the subspace of ${\cal H}$ orthogonal to $x$) which satisfies $\mu_{\hat x} (B_{\hat x}) \geq 2 \epsilon$.  By Fubini's theorem, then,
\begin{equation} I^{c_{k-1} r_0 \delta'}_k (\mu_0, \mu, \ldots, \mu) \leq 2 C_{k-1} \epsilon^{2} \label{fubini1} \end{equation}
(where $\delta'$ may now be any quantity which is smaller than the $(k-1)$-content of $B_{\hat x}$ for any pair of $x$ and $B_{\hat x}$ with $\mu_{\hat x} (B_{\hat{x}}) \geq 2 \epsilon$ and $x \neq 0$) by virtue of \eqref{volumefactor} and the fact that $||x|| \geq r_0$ on the support of $\mu_0$.

Consider the inverse image $P_{\hat x}^{-1} (B_{\hat x})$; for convenience, call this set simply $B$.  This $B$ is itself a centered ellipsoid in ${\cal H}$ with infinite extent in the direction of $x$.
  The intersection of $B$ with the ball $\set{x}{||x|| \leq r_0}$ is contained in the centered ellipsoid given by $\tilde B := \set{ x \in {\cal}}{\sum_i \left( \frac{1}{2 \ell_i^2} + \frac{1}{2 r_0^2} \right) |\left<x, {\omega}_i\right>|^2 \leq 1}$, where $\ell_0 = \infty$ and $\omega_0 = \hat{x}$ (and the remaining $\ell_i$ are the extents of $B$ in all directions orthogonal to $\hat{x}$).  Since $\mu(B) \geq 2 \epsilon$, it follows that $\mu(B \cap \set{x}{|x| \leq r_0}) \geq \epsilon$; thus the $\mu$-measure of $\tilde B$ is at least $\epsilon$ as well.  
Now $( (2 \ell_i^2)^{-1} + (2 r_0^2)^{-1} )^{-1} \leq \min \{2 \ell_i^2, 2 r_0^2 \}$.  Consequently,  the $k$-content of $\tilde B$ is at most $2^k r_0$ times the $k-1$ content of $B_{\hat x}$.  Thus, if $\delta$ is taken to be the infimum of the $k$-contents of all centered ellipsoids $B$ with $\mu(B) \geq \epsilon$, then $\delta \leq 2^{-k} r_0 \delta'$ for some $\delta'$ satisfying \eqref{fubini1}; thus it follows that
\[ I^{c_{k} \delta}_{k} (\mu_0, \mu, \ldots, \mu) \leq 2 C_{k-1}  \epsilon^{2} \]
when $c_k = 2^{-k} c_{k-1}$.
By multilinearity, it must also be the case that
\begin{equation}
 I_{k}^{c_k \delta} (\mu,\mu,\ldots,\mu) \leq 2 k C_{k-1} \epsilon^{2} +  I_{k}^{c_k \delta} (\mu_1,\mu_1,\ldots,\mu_1)\label{decomp1}
\end{equation}
where $\mu_1 := \mu - \mu_0$ is a nonnegative Borel measure of mass $1 - \epsilon$.

Now $\frac{1}{1-\epsilon} \mu_1$ is also a Borel probability measure; moreover, any centered ellipsoid $B$ with measure at least $\frac{\epsilon}{1-\epsilon}$ with respect to this measure will have $\mu$ measure at least $\epsilon$.  Thus, taking a supremum of both sides of \eqref{decomp1} with respect to $\mu$ and $\theta$ will give
\[ {\cal C}_{k,\epsilon} \leq 2 k C_{k-1} \epsilon^2 + (1-\epsilon)^k {\cal C}_{k,\frac{\epsilon}{1-\epsilon}}. \]
Bootstrapping this inequality gives, for any positive integer $j$, 
\begin{align*}
 \frac{{\cal C}_{k,\epsilon} }{2k C_{k-1}}   \leq  & \epsilon^2 + (1-\epsilon)^k \left( \frac{\epsilon}{1-\epsilon} \right)^2 + (1-2\epsilon)^k \left( \frac{\epsilon}{1-2\epsilon} \right)^2 + \cdots \\
& + (1-(j-1) \epsilon)^k \left( \frac{\epsilon}{1- (j-1)\epsilon} \right)^2 + (1-j \epsilon)^k {\cal C}_{k,\frac{\epsilon}{1 - j \epsilon}}.
\end{align*}
Choose $j$ so that $0 \leq 1 - j \epsilon \leq \epsilon$.  Since ${\cal C}_{k,\epsilon'} \leq 1$ and $k \geq 2$, it follows that
\[\frac{{\cal C}_{k,\epsilon} }{2k C_{k-1}} \leq j \epsilon^2 + \epsilon^k \leq 2 \epsilon \]
(since we may assume $\epsilon \leq 1$).  Thus, the induction hypothesis holds for $C_k = 4 k C_{k-1}$, giving $C_k = 4^{k-1} k!$.
\end{proof}

It is perhaps worth noting that the proof provided could be used to establish better decay rates (as $\epsilon \rightarrow 0$) for $I^{c_k \delta}_k$ for restricted classes of measures (all that is necessary is that the class is closed under the cutoff/renormalization procedure).  {\it A priori} the bootstrapping argument would continue to work for any decay rate $\epsilon^\sigma$ for $\sigma < k$.  For the class of all Borel probability measures, though, theorem \ref{mainst} is sharp (that is, no better decay rate can hold).



The next step is to observe that the same sort of statement can be made about the fully multilinear functional $I^{\delta}_k(\mu_1,\ldots,\mu_k)$ (which is crucial since \eqref{multi} must, in particular, hold for $L^1$-normalized characteristic functions):
\begin{corollary}
For any positive integer $k$, there exist $c_k$ and $C_k$ such that, for any probability measures $\mu_1,\ldots,\mu_k$ on ${\cal H}$ and any $\epsilon > 0$, 
\[I^{c_k \delta}_{k}(\mu_1,\ldots,\mu_k) \leq C_k \epsilon \]
provided $\delta \leq |B_1|_k^{1/k} \cdots |B_k|_k^{1/k}$ for all centered ellipsoids $B_1,\ldots,B_k$ satisfying $\mu_i(B_i) \geq \epsilon$.  \label{maincor}
\end{corollary}
\begin{proof}
For any probability measure $\mu$ on ${\cal H}$ and any $a > 0$, let $\mu^a$ be defined by
\[ \int f(x) d \mu^a(x) := \int f(a^{-1} x) d \mu(x), \]
that is, $\mu^a$ is an isotropic dilation of $\mu$ (the support of $\mu^a$ is $a$ times the support of $\mu$).  The homogeneity of the standard of the determinant guarantees that $I^\delta_{k}(\mu_1^{a_1},\ldots,\mu_k^{a_k}) = I^{ a_1\cdots a_k \delta}_{k}(\mu_1,\ldots,\mu_k)$.
For a given probability measure $\mu_i$, let $a_i$ be the infimum of $|B_i|_k^{1/k}$ as $B_i$ ranges over all centered ellipsoids $B_i$ satisfying $\mu_i(B_i) \geq \epsilon$.  With this scaling, every centered ellipsoid $B$ for which $\mu_i^{a_i}(B) \geq \epsilon$ must have $k$-content at least equal to $1$.  Now consider the measure $\mu := \frac{1}{k} \left( \mu_1^{a_1} + \cdots + \mu_k^{a_k} \right)$.
If $B$ is a centered ellipsoid for which $\mu(B) \geq \epsilon$, then there must also be some index $i$ for which $\mu_i^{a_i}(B) \geq \epsilon$; hence $B$ must have $k$-content at least equal to one.
Thus it follows that
\begin{align*}
I^{c_k a_1 \cdots a_k}_{k}(\mu_1,\ldots,\mu_k)  = I^{c_k}_{k} (\mu_1^{a_1}, \ldots, \mu_d^{a_d})  \leq \frac{k^k}{k!} I^{c_k}_{k} (\mu,\ldots,\mu) \leq C_k \frac{k^k}{k!} \epsilon,
\end{align*}
which gives precisely the desired inequality.
\end{proof}
The necessary preparations now being complete, the restricted weak-type analog of theorem \ref{equivtheorem} can now be established for the $k$-linear functionals
\[ \widetilde{T}^{-\gamma}_{\mu,k} (f_1,\ldots,f_k) := \int \cdots \int \prod_{j=1}^k f_j(y_j) \frac{d \mu(y_1) \cdots d \mu (y_k)}{(\det(0,y_1,\ldots,y_k))^{\gamma}} \]
under the assumption that $\mu$ is $k$-admissible.  It is perhaps worth noting that, in general, one can expect only restricted weak-type inequalities to hold for these functionals, unless it is {\it a priori} known that $\mu$ is compactly supported away from $0$ (which is precisely the case in the work of Drury \cite{drury1988}).  However, by Fubini's theorem and Christ's multilinear trick, the restricted weak-type estimates are sufficient to imply theorem \ref{equivtheorem}.
\begin{theorem}
When $\mu$ is $k$-admissible and $\alpha$ is any positive real number, the following are equivalent:
\begin{enumerate}
\item The measure $\mu$ is $k$-curved at the origin with exponent $\alpha$.
\item For all $\gamma \in (0,\alpha)$, there is a constant $C < \infty$ such that
\begin{equation}
\widetilde{T}^{-\gamma}_{\mu,k} ( \chi_{E_1},\ldots,\chi_{E_k}) \leq C \prod_{j=1}^{k} \mu(E_j)^{1 - \frac{\gamma}{k \alpha}} \label{rwt1}
\end{equation}
for all measurable sets $E_1,\ldots,E_k$.
\item There exists a $\gamma \in (0,\infty)$ and a $C < \infty$ such that \eqref{rwt1} holds.
\item For all $\gamma \in (0,\infty)$, there exists a constant $C > 0$ such that
\begin{equation}
\widetilde{T}^{\gamma}_{\mu,k} ( \chi_{E_1},\ldots,\chi_{E_k}) \geq C \prod_{j=1}^k \mu(E_j)^{1 + \frac{\gamma}{k \alpha}} \label{rwt2}
\end{equation}
for all measurable sets $E_1,\ldots,E_k$.
\item There exists a $\gamma \in (0,\infty)$ and a $C > 0$ such that the inequality \eqref{rwt2} holds.
\end{enumerate}
\end{theorem}
\begin{proof}
Conditions 2 and 4 immediately imply conditions 3 and 5, respectively.  Likewise, conditions 2 and 3 imply conditions 4 and 5, respectively, by means of Cauchy-Schwartz:
\[ \prod_{j=1}^k \mu(E_j)^2 \leq \left( \int_{E_1 \times \cdots \times E_k} \det(0,\cdot)^{\gamma} d \mu^k \right) \left( \int_{E_1 \times \cdots \times E_k} \det(0,\cdot)^{-\gamma} d \mu^k \right) \]
(the $k$-admissibility of $\mu$ guarantees that $1 = \det(0,\cdot)^{\gamma} \det(0,\cdot)^{-\gamma}$ $\mu$-a.e.).

To show that condition 5 implies condition 1, it suffices to test the integral in the situation where $E_1 = \cdots = E_k = B$ for any centered ellipsoid $B$.  The conclusion will follow immediately once it can be shown that
\begin{equation} \det(0,y_1,\ldots,y_k) \leq C_k |B|_k \label{fivetoone}
\end{equation}
for $(y_1,\ldots,y_k) \in B \times \cdots \times B$.
Given the centered ellipsoid $B$, and let $\omega_1,\ldots,\omega_k$ be chosen so that the corresponding lengths $\ell_1,\ldots,\ell_k$ satisfy $\ell_1 \cdots \ell_k \geq \frac{1}{2} |B|_k$ (it suffices to assume that $|B|_k$ is finite).  It follows that all other lengths of $B$ in the remaining directions are no longer than twice the minimum of $\ell_1,\ldots,\ell_k$.  In particular, this means that for any vector in the ellipsoid, its length will be less than some fixed constant times $\rho := \min_{i=1,\ldots,k} \ell_i$ after it is projected onto the subspace orthogonal to $\omega_1,\ldots,\omega_k$.
For any vector $y$, let $y^{(j)} = \left< y, \omega_j \right> \omega_j$ for $j=1,\ldots,k$ and $y^{(k+1)} = y - \sum_{j=1}^{k} y^{(j)}$.  The definition of the $\det$ function guarantees that
\begin{equation} \det(0,y_1,\ldots,y_k)^2 = \sum_{\beta} \det \left( \left< y, y^{(\beta)} \right> \right) \label{expandi}
\end{equation}
where $\beta$ ranges over all multiindices of length $k$ with entries in $\{1,\ldots,k+1\}$ and $\det (\left<y,y^{(\beta)} \right>)$ denotes the standard $k \times k$ determinant of the matrix whose $(i,j)$-entry is the inner product of $y_i$ with $y_j^{(\beta_j)}$.  If any two entries of the multiindex $\beta$ have the same value and that value is one of $\{1,\ldots,k\}$, the determinant will necessarily be zero (since there will be linearly dependent columns in the matrix).  Thus it may be assumed that the entries of $\beta$ must be distinct unless they happen to equal $k+1$.  Furthermore, if some entry of $\beta$ takes the value $i$ for $i \neq k+1$, then the Euclidean length of the corresponding column in $\R^{k}$ is at most $\sqrt{k} \ell_i^2$ (since the entries will be products of the form $\left<y_l,\omega_i \right> \left<y_{l'},\omega_i \right>$).  If a particular index takes the value $i=k+1$, then the Euclidean norm in $\R^k$ of the corresponding column is at most $C_k \rho^2$ since 
\[\left<y_l, y_{l'}^{(k+1)} \right> = \left< y_l^{(k+1)},y_{l'}^{(k+1)} \right>. \]
Consequently, any single term on the right-hand side of \eqref{expandi} is bounded above by a constant depending on $k$ only times the product $(\ell_1 \cdots \ell_k)^2$.  This gives precisely \eqref{fivetoone}.

Finally, consider the statement that condition 1 implies condition 2.  Let $||\mu||_{0}$ be the smallest value of $C_\alpha$ such that \eqref{curved} holds for all centered ellipsoids (for fixed values of $k$ and $\alpha$).
Let $\mu_E$ be the restriction of $\mu$ to the set $E$, renormalized to have mass $1$.  Suppose $B$ be a centered ellipsoid with $\mu_E(B) \geq \epsilon$.  It follows that $\epsilon \mu(E) \leq \mu(E \cap B) \leq \mu(B) \leq ||\mu||_0 |B|_k^\alpha$.  Hence, for any sets $E_1,\ldots, E_k$ and any $\epsilon > 0$, corollary \ref{maincor} guarantees that
\[I^{c_k \delta}_{k}(\mu_{E_1},\ldots,\mu_{E_k}) \leq C_k \epsilon \]
provided $\delta \leq  ||\mu||_0^{-1/\alpha} \epsilon^{1/\alpha} \prod_{i=1}^k \mu(E_i)^{1/(k\alpha)}$.  If $\delta$ is considered fixed and the minimal $\epsilon$ satisfying this inequality is chosen, it follows that, for any $\delta > 0$, 
\[I^{\delta}_{k}(\mu_{E_1},\ldots,\mu_{E_k}) \leq  C_k ||\mu_0|| \delta^{\alpha} \prod_{j=1}^k \mu(E_j)^{-\frac{1}{k}}. \]
Hence the following two estimates hold:
\begin{align*}
 \int_{E_1 \times \cdots \times E_k} 2^{-\gamma l} \chi_{\det(0,\cdot) \sim 2^{l}} d \mu^{k} & \leq C_{k} ||\mu||_0 2^{(\alpha - \gamma) l} \prod_{j=1}^k \mu(E_j)^{1-\frac{1}{k}}, \\
\int_{E_1 \times \cdots \times E_k} 2^{-\gamma l} \chi_{\det(0,\cdot) \sim 2^{l}} d \mu^{k} & \leq  2^{- \gamma l} \prod_{j=1}^k \mu(E_j). 
\end{align*}
Fix $\gamma$ satisfying $\alpha > \gamma$.  Condition 2 follows by summing over $l$ using the first inequality for $l \geq l_0$ and the second when $l < l_0$ and minimizing the result as a function of $l_0$.  In particular, \eqref{rwt1} holds with $C = C_{k,\alpha,\gamma} ||\mu||_0^{\gamma/\alpha}$.
\end{proof}

\section{Connections to other work}
\label{conditionsec}
This final section is devoted to the illustration of several important connections between theorem \eqref{equivtheorem} and earlier work.  The first to be addressed is the question of Gaussian extremizability.  It has been known for some time that many of the most important geometrically-motivated integral inequalities, Young's inequality for convolutions and Brascamp-Lieb inequalities in general, have Gaussian extremals.  This result is originally due to Lieb \cite{lieb1990}, and has motivated related work the context of heat-flow monotonicity by  Bennett, Bez, Carbery, and Hundertmark  \cite{bbch2008} and Bennett, Bez, and Carbery \cite{bbc2008} as well as work by Bennett and Bez on nonlinear Brascamp-Lieb inequalities \cite{bb2009}.  While Gaussian functions cannot be expected to be perfect extremizers of the inequality \eqref{multi} in the general case, it nevertheless suffices to test only on Gaussians:
\begin{theorem}
Suppose $\mu$ is a $k$-admissible measure on $\R^d$ for some $d$.  For fixed values $0 < \alpha < \gamma$, the inequality \eqref{rwt1} holds for some constant $C$ uniformly in the sets $E_1,\ldots,E_k$ if and only if there is a constant $C' < \infty$ such that
\begin{equation} \int \cdots \int \prod_{i=1}^k f(y_i)  \frac{d \mu(y_1) \cdots d \mu(y_k)}{ (\det(0,y_1,\ldots,y_k))^{\gamma}}  \leq C' \left( \int f(y) d \mu(y) \right)^{k - \frac{\gamma}{\alpha}} \label{gauss1} \end{equation}
uniformly for all $f(y) := e^{-||Q y||^2}$, where $Q$ is any $d \times d$ real matrix.  Moreover, for exponents $p_1,\ldots,p_{k+1}$ satisfying $\frac{1}{p_i'} < \frac{\gamma}{k \alpha}$ and $\sum_{i=1}^{k+1} \frac{1}{p_i'} = \frac{\gamma}{\alpha}$, the inequality \eqref{multi} holds for some constant $C < \infty$ if and only if there exists a constant $C'$ such that
\begin{equation} \int \cdots \int \prod_{i=1}^k f(y_i) \frac{d \mu(y_1) \cdots d \mu(y_k) d \mu(y_{k+1})}{ (\det(y_1,\ldots,y_k,y_{k+1}))^{\gamma}}  \leq C' \prod_{i=1}^{k+1} ||f||_{L^{p_i}(\mu)} \label{gauss2} \end{equation}
uniformly for all $f(y) := e^{- ||Q y - y_0||^2}$ where $Q$ is any real $d \times d$ matrix and $y_0 \in \R^d$.
\end{theorem}
For any $d \times d$ real matrix $Q$, let $|Q|_k$ equal the reciprocal of the product of the $k$ smallest eigenvalues of $(Q^* Q)^{1/2}$.  By virtue of \eqref{fivetoone} and the homogeneity of $\det(0,y_1,\ldots,y_k)$, it must be the case that
\[ \det(0,y_1,\ldots,y_k) \leq C_k |Q|_k \prod_{j=1}^k ||Q y_j|| \]
since $|Q|_k$ simply equals (up to a constant) the $k$-content of the ellipsoid given by the sublevel set $\set{x \in \R^d}{||Q x||^2 \leq 1}$.  If \eqref{gauss1} holds, for example, then the inequality just established for the $\det$ function guarantees that
\[ C_k^{-\gamma} |Q|_k^{-\gamma} \left( \int ||Q y||^{-\gamma} e^{-||Q y||^2} d \mu(y) \right)^k \leq C' \left( e^{- ||Q y||^2} d \mu(y) \right)^{k - \frac{\gamma}{\alpha}}. \]
Now $||x||^{-\gamma} e^{\epsilon ||x||^2} \geq C_{\gamma,\epsilon}$ uniformly in $||x||$, so it must ultimately be the case that
\[ \int e^{-(1-\epsilon) || Q y||^2} d \mu(y) \leq C_{k,\gamma,\epsilon} C'^{\frac{\alpha}{\gamma}} |Q|_k^{\alpha} \]
(since, on the right-hand side, $e^{- ||Q y||^2}$ may be replaced with $e^{-(1-\epsilon)||Q y||^2}$ for free).
As for \eqref{gauss2}, the inequality $\det(y_1,\ldots,y_{k+1}) \leq C_k |Q|_k \prod_{j=2}^{k+1} ||Q y_j - Q y_1||$ along with the inequality $\prod_{j=2}^{k+1} ||x_j - x_1||^{-\gamma} \prod_{j=1}^{k+1} e^{\epsilon ||x_j - y_0||^2} \geq C_{\gamma,\epsilon,k}$ for all $y_0,x_1,\ldots,x_{k+1}$ give 
\[ \int e^{-(1-\epsilon)||Q y - y_0||^2} d \mu(y) \leq C_{k,\gamma,\epsilon} C'^{\frac{\alpha}{\gamma}} |Q|_k^{\alpha} \]
in just the same way that the corresponding inequality was derived in the case of \eqref{gauss1}.  The proof of the theorem is then complete once the following proposition is established:
\begin{proposition}
Suppose $\mu$ is a $k$-admissible measure on $\R^d$.  Then $\mu$ is $k$-curved with exponent $\alpha$ if and only if there exists a constant $C < \infty$ such that
\begin{equation}
 \int e^{- ||Q x-x_0||^2} d \mu(x) \leq C |Q|_k^{\alpha} \label{gauss0}
\end{equation}
for all $d \times d$ real matrices $Q$ and all $x_0 \in \R^d$.  In addition, $\mu$ is $k$-curved at the origin if and only if \eqref{gauss0} holds for all $Q$ when $x_0 = 0$.
\end{proposition}
\begin{proof}
Since every centered ellipsoid $B$ may be realized as a sublevel set of the form $\set{x \in \R^d}{||Q x||^2 \leq 1}$ for some $Q$ (and vice-versa), the proposition follows immediately from the inequalities:
\begin{align*}
 e^{-1} \mu( \set{x \in \R^d}{ ||Q x||^2 \leq 1} ) \leq  & \int e^{- ||Q x||^2} d \mu(x) \\
& = 2 \int_0^\infty t e^{-t^2} \mu( \set{x \in \R^d}{ ||Q x||^2 \leq t} ) dt. 
\end{align*}
The first inequality guarantees that the $\mu$-measure of the ellipsoid will be controlled by the integral of the Gaussian, and the second controls the Gaussian integral under the assumption that $\mu( \set{x \in \R^d}{||Q x||^2 \leq t}) \leq C_\alpha t^{k \alpha} |Q|_k^{\alpha}$ (which follows from the curvature of $\mu$).  The non-centered case follows immediately by translation.
\end{proof}

To conclude, it is worth noting that many of the geometric conditions on curves and hypersurfaces that have appeared in the literature of oscillatory integrals and geometric averaging operators are themselves sufficient conditions to establish $k$-curvature.  The curvature conditions identified by Nagel, Seeger, and Wainger \cite{nsw1993} as well as Iosevich and Sawyer \cite{is1997} in the context of $L^p$ bounds for maximal averages are two such examples.  In these cases, the central quantity of study was, roughly speaking, the growth rate of the distance from a hypersurface to a specified tangent plane (or, more precisely, some average growth rate over all tangent planes).  Such conditions are sufficient to prove an inequality of the form \eqref{curved}, but are, in general, not necessary.  Two examples are included below for completeness:
\begin{proposition}
Suppose $\mu$ is $k$-admissible. For any positive integer $k$ and positive $\alpha$, let
\[ F_{k,\alpha}(y) := \sup_{B} \frac{ \mu(y + B)}{|B|_k^{\alpha}} \]
where $B$ ranges over all centered ellipsoids.  Then for any $p \in (0,\infty)$,
\[ || F_{k,\frac{\alpha p}{p+1}} ||_{L^\infty(\mu)} \leq 2^{\alpha k} ||F_{k,\alpha}||_{p,\infty}^{\frac{p}{p+1}} \]
(here $||f||_{p,\infty}^p := \sup_{\lambda > 0} \lambda^p \mu (\set{ y \in {\cal H}}{ |f(y)| > \lambda})$ is the usual weak-$L^p$ norm).  In particular, if $F_{k,\alpha}$ is in weak-$L^p(\mu)$, then $\mu$ is $k$-curved with exponent $\frac{\alpha p}{p+1}$.
\end{proposition}
\begin{proof}
If $F_{k,\alpha}(y') \leq \lambda$ at any point $y' \in y + B$, then $\mu(y + B) \leq 2^{\alpha k} \lambda |B|_k^{\alpha}$ for the simple reason that $y' + 2 B \supset y + B$ in this case.  Thus, for any $\lambda > 0$, either $\mu(y + B) \leq 2^{\alpha k} \lambda |B|_k^{\alpha}$ or $y + B \subset \set{y \in {\cal H}} {|F_{k,\alpha}(y)| > \lambda}$, in which case $\mu(y+B) \leq ||F_{k,\alpha}||_{p,\infty}^p \lambda^{-p}$.  Minimizing over $\lambda$ gives the proposition.
\end{proof}

\begin{proposition}
Suppose $\mu$ is $k$-admissible.  If there exists a constant $C < \infty$ and an exponent $\alpha > 0$ such that
\[ \mu \left( \set{ y \in {\cal H}}{\mathop{\mathrm{dist}}(y,{\cal H}_0) \leq \delta} \right) \leq C \delta^{\alpha k} \]
for all $\delta > 0$ and all affine $(k-1)$-dimensional subspaces ${\cal H}_0 \subset {\cal H}$ (and $\mathop{\mathrm{dist}}(y, {\cal H}_0)$ is the distance from $y$ to ${\cal H}_0$), then $\mu(B) \leq C |B|_k^{\alpha}$ for all ellipsoids $B$.  The same inequality will hold for all centered ellipsoids when ${\cal H}_0$ ranges over subspaces passing through the origin.
\end{proposition}
\begin{proof}
Suppose $B := \set{x \in {\cal H}}{ \sum_i \ell_i^{-2} |\left<x - x_0,\omega_i \right>|^2 \leq 1}$ where the $\omega_i$'s are an orthonormal basis and $\ell_i$ is a nonincreasing function of $i$.  Let ${\cal H}_0$ be the affine subspace $\set{x_0 + \sum_{i=1}^{k-1} a_i \omega_i}{a_1,\ldots,a_{k-1} \in \R}$.  For any $y \in B$,
\[ (\mathop{\mathrm{dist}}(y,{\cal H}_0))^{2} =  \sum_{i=k}^\infty |\left<y - x_0,\omega_i \right>|^2 \leq \ell_{k}^2 \sum_{i=k}^\infty \ell_i^{-2} |\left<y - x_0,\omega_i \right>|^2 \leq \ell_k^2. \]
Thus $ \mu(B) \leq C \ell_k^{\alpha k}$. Since $\ell_k^k \leq |B|_k$, the proposition must hold.
\end{proof}

\bibliography{mybib}

\end{document}